\documentclass[a4paper,11pt]{article}
\usepackage[colorlinks,allcolors=NavyBlue]{hyperref}
\usepackage{graphicx} 
\usepackage[top=20mm,bottom=20mm,left=20mm,right=20mm]{geometry}
\usepackage{amsmath,amssymb,amsthm}
\usepackage{mathtools} 
\usepackage{lmodern,microtype} 
\usepackage[T1]{fontenc} 
\usepackage[utf8]{inputenc}
\usepackage{titlesec} 
\usepackage[dvipsnames]{xcolor} 
\usepackage{comment}
\usepackage{authblk}
\usepackage{orcidlink}
\usepackage[normalem]{ulem}

\newcommand{\red}[1]{#1}\newcommand{\blue}[1]{#1}

\title{A critical threshold for \blue{the cosmological Euler-Poisson system}}

\author[D.~Fajman\orcidlink{0000-0003-3034-6232}, M.~Maliborski\orcidlink{0000-0002-8621-9761}, M.~Ofner\orcidlink{0009-0009-2271-5211}, T.~Oliynyk\orcidlink{0000-0003-3457-9578}, Z.~Wyatt\orcidlink{0000-0001-5120-1839}]{David Fajman, Maciej Maliborski, Maximilian Ofner, Todd Oliynyk, Zoe Wyatt}

\date{\today}

\titleformat{\section}[hang]{\centering \large\sc}{\thesection}{1em}{}[]

\titleformat{\subsection}[hang]{ \bfseries}{\thesubsection}{1em}{}[]

\newcommand{\laplacian}{\Delta}

\newtheoremstyle{lemma}{\topsep}{\topsep}{\itshape}{}{\sc}{.}{5pt}{}
\theoremstyle{lemma}
\newtheorem{lemma}{Lemma}
\numberwithin{lemma}{section}

\theoremstyle{definition}
\newtheorem{definition}{Definition}

\theoremstyle{plain}
\newtheorem{theorem}{Theorem}

\theoremstyle{remark}
\newtheorem{remark}{Remark}

\usepackage{csquotes}
\usepackage[
url=false,
backend=biber,
style=numeric,
giveninits=true,
doi=true,
backref=false,
date=year,
maxbibnames=99,
eprint=false
]{biblatex}

\bibliography{poisson}

\begin{document}
	\maketitle
	
	\begin{abstract}

We consider the gravitational Euler-Poisson system with a linear equation of state on an expanding cosmological model of the Universe.  The expansion of the spatial sections  introduces an additional dissipating \blue{effect} in the Euler equation. We prescribe the expansion rate of space by a scale factor $a(t)=t^\alpha$ with $\alpha\in(0,1)$, which describes the growth of length scales over time. This model is regularly applied in cosmology to study classical fluids in an expanding Universe.

We study the behaviour of solutions to this system arising from small, near-homogeneous initial data and discover a \emph{critical} change of behaviour near the expansion rate $\alpha=2/3$, which corresponds to the matter-dominated regime in cosmology.
In particular, we prove that for $\alpha>2/3$ the fluid variables are global in time and remain small provided they are sufficiently small in a suitable norm initially. In the complementary regime $\alpha\leq2/3$, we present numerical evidence for shock formation of solutions to the Euler equation for arbitrarily small initial data. 
In combination, this establishes the existence of a critical stability threshold for barotropic fluids in expanding domains. In contrast to our previous work on the corresponding relativistic system \cite{fajman2024arxiv}, the threshold in the classical system considered here is independent of the speed of sound of the fluid. This establishes that fluids in cosmology behave fundamentally different in the non-relativistic regime than in the relativistic one.
\end{abstract}

\section{Introduction}

The Euler-Poisson system, with an attractive gravitational force, describes the dynamics of self-gravitating fluids in astrophysical and cosmological settings  (e.g.~\cite{BT2008}, Appendix F). On cosmological scales, fluids are used to model the dynamics in the post-inflationary Universe  (e.g.~\cite{Baumann}). In Newtonian cosmology, the trajectory of an observer that is co-moving with expansion is given by $r(t,x)=a(t)x$, and in such coordinates, the Euler-Poisson system reads
\begin{equation}\label{introeq}
\begin{aligned}
\partial_t \rho + \nabla \cdot (\rho \mathbf{u}) +3\frac{\dot a(t)}{a(t)}\rho&= 0, \\
(\partial_t  + \mathbf{u} \cdot \nabla) \mathbf{u} +2\frac{\dot a(t)}{a(t)}\mathbf u + a(t)^{-2} \rho^{-1}\nabla p &= -a^{-2}(t)\nabla \Phi, \\
\Delta \Phi &= -4\pi a(t)^2(\rho-\overline \rho),
\end{aligned}
\end{equation}
for the energy density $\rho$ and \red{peculiar} velocity field $\mathbf u$ of the fluid and gravitational potential $\Phi$, where the sign of the potential term in the Euler equation and Poisson equation model an attractive force \cite{Baumann}. This system can be derived from the Newton-Cartan-Ehlers model of gravity, as we show below, or alternatively, by taking the Newtonian limit \footnote{Here, $c$ is the speed of light and $s$ is a characteristic speed of the gravitating matter.} $s/c \searrow 0$ of the Einstein-Euler equations \cite{Oliynyk:CMP_2010,Oliynyk:JHDE_2010}.  \\

A key goal in cosmological settings is to use the system \eqref{introeq} to model structure formation, by which matter concentrates in certain regions and thereby provides the seeds for the observable matter distribution in the current Universe.
For example, the linearized version of \eqref{introeq} leads to the well-known Jeans instability which predicts growth rates for the fluid density contrast $\hat{\rho} = \frac{\rho - \overline\rho}{\overline\rho}$ \cite{Baumann}.

The system \eqref{introeq} can be used to model the onset of structure formation from an almost homogeneous initial configuration precisely because of the tendancy for fluids to develop shocks in finite time. Shocks correspond to discontinuities in the fluid variables. On the level of the mass density this corresponds to discontinuities in the density contrast $\hat{\rho}$, which resemble the barriers between low density and high density regions. In the subsequent evolution the overdense regions eventually collapse under gravity to compact objects, which form early states of the structures observed today. 

\subsection{Global existence and shock formation in the Euler-Poisson system}

In this paper, we  consider the system \eqref{introeq} on an expanding background with \red{scale factor $a(t)=t^{\alpha}$}, $\alpha>0$, which generates specific dissipative terms and \blue{decaying} factors in some of the terms (see \eqref{eq:systemrhou}). In accordance with popular numerical schemes in cosmology (see, e.g., \cite{adamek2016}), we fix the scale factor as a function of time. This is in contrast to a fully relativistic treatment, in which expansion is influenced locally by the matter distribution. While this fixes some of the gravitational degrees of freedom, it is important to note that we make \textbf{no linear approximations} in the dynamics of the fluid equations, which is in contrast to the typical linearisation argument used when deriving the Jeans instability.

We study the evolution of solutions to  \eqref{introeq}  with initial data close to a homogeneous fluid state
\begin{equation}
\rho_{\mathrm{hom}}(t=1)=\rho_{\mathrm{c}}\in\mathbb R_+,\quad \vec v_\mathrm{hom}(t=1)=\vec 0,
\end{equation}
which has the time-evolution
\begin{equation}
\rho_{\mathrm{hom}}(t)=\rho_{\mathrm{c}}\cdot t^{-3\alpha},\quad \vec v_\mathrm{hom}(t)=\vec 0.
\end{equation}

We consider the standard fluid equation of state in cosmology, which is that of a barotropic fluid (cf.~\cite{BT2008})
\begin{equation}
p=K\rho,
\end{equation}
where $K=c_s^2$ is the square of the speed of sound. \blue{To understand the conditions that could source structures, we aim on the one hand to establish parameters $(\alpha, K)$ for which sufficiently small initial data yields solutions which remain small over time and hence avoid shock formation (stable regime). On the other hand, we also aim to determine for which parameters arbitrarily small initial data leads to shock formation in finite time (unstable regime).} The latter can be associated with the regime of structure formation.

We study this question in this paper by means of analytical and numerical approaches for the shock \blue{suppression} and the shock formation, respectively. 

In the main theorem, Theorem \ref{thm1}, we prove that for an expansion rate $\alpha>2/3$, there exists an open neighborhood in a suitable Sobolev space of initial data such that the corresponding solutions remain close to the background solution for all times. The proof is based on energy estimates for expansion-normalized fluid variables, which use a correction method. Standard Sobolev norms are complemented by small indefinite terms, in order for the corrected norm to fulfill suitable decay estimates. These estimates are sufficiently strong to establish global smallness of the appropriately rescaled solutions in the small data regime. 

In Section \ref{sec:numerics}, we present a \red{complementary} numerical study in combination with a subsequent scaling analysis to provide strong evidence for shock formation for arbitrarily small initial data whenever $\alpha\leq 2/3$. Moreover, the scaling analysis reveals that the shock formation from arbitrarily small data gets weaker as the expansion rate grows and asymptotically fails when it approaches $\alpha=2/3$ from below. Note that, for the sake of computational efficiency, our numerical experiments are limited to one spatial dimension and neglect the self-gravitating nature of the fluid. This approach is justified by the observation that, in the analytical proof of Theorem \ref{thm1}, the terms that stem from the gravitational interaction are, in fact, error terms. 

The range of expansion rates that are covered by our analysis include all relevant examples of cosmological models in the decelerated regime. These are radiation-dominated models $\alpha=1/2$, matter dominated models $\alpha=2/3$ and all models with expansion rates between those values. This class of models arise as solutions to the Friedman equations with fluid matter sources \cite{cqgreview}. Moreover, the faster expanding models with $\alpha\in(2/3,1)$ are solutions to the Friedman equation with suitable scalar field sources.

In combination, these results imply the existence of a critical threshold for barotropic fluids at the expansion rate $\alpha=2/3$, which demarcates the regimes of structure formation and \blue{homogenisation} in the cosmological Euler-Poisson system, respectively.
\red{We point out that this critical expansion rate coincides with the matter-dominated Universe, which describes the identical epoch in the cosmological evolution.}

\subsection{Relation to the relativistic regime}
The avoidance of the formation of shocks in a fluid due to sufficiently fast expansion of space is referred to as \emph{fluid stabilization}. It was discovered by Brauer et al.~\cite{brauer1994} for the classical Euler-Poisson system in a spacetime undergoing accelerated expansion ($\alpha>1$), but then extensively studied for the relativistic Euler equations coupled to the Einstein equations (cf.~\cite{cqgreview} for a review). While it has been shown that accelerated and linear expansion always leads to fluid stabilization in subradiative fluids \cite{rodnianski2013,hadzic2015,speck2013,friedrich2017,fajman2024arma,fajman2024imrn},
in the regime of decelerated expansion (which corresponds to $\alpha<1$) relativistic fluids may form shocks from arbitrarily small initial inhomogeneities. Note that in the superradiative regime, $\frac{1}{3}<K<1$, accelerated expansion does not guarantee stability, as seen, e.g., in \cite{oliynyk2024cmp,beyer2023}. A critical phenomenon similar to the one presented in the present paper was discovered for the relativistic Euler equations in \cite{fajman2024arxiv} and \cite{fajman2025arxiv}. A crucial difference between the results lies in the fact that, in the relativistic case, the stability threshold is dependent on the speed of sound $c_S$ of the fluid. In particular, the critical expansion rate in the relativistic case is given by $\alpha_{\mathrm{crit}}=\tfrac2{3(1-K)}$.

 As shown in the present paper, in the non-relativistic case the threshold is universal and given by $\alpha_{\mathrm{crit}}=2/3$. This observation hints at the fact that the nature of structure formation process is connected to the characteristic speed of matter at hand and plays out fundamentally different in the relativistic and non-relativistic regimes.

\subsection*{Acknowledgements}
D.F.~and M.O.~acknowledge support by the Austrian Science Fund (FWF) grant \emph{Matter dominated Cosmology} 10.55776/PAT7614324. Furthermore, M.O.~recognizes that this work was in part funded within the Dimitrov Fellowship Program of the OeAW. The research of M.M.~was supported by the FWF through project 10.55776/P36455 and the START-Project 10.55776/Y963.

	\section{Equations of motion in Newtonian cosmology}
	
	We start with a short overview of the necessary notations and conventions regarding function spaces. 
	
	\subsection{Notation and norms}
	Throughout this article we assume that integration is performed over the manifold $(\mathbb{T}^{3},g_{E}=\delta_{ij}dx^{i}dx^{j})$, i.e., for a function $f:\mathbb{T}^{3}\to \mathbb{R}$, we simply write
	\begin{equation*}
		\int f\coloneqq \int_{\mathbb{T}^{3}}f\coloneqq \int_{\mathbb{T}^{3}}f dV_{g_{E}},
	\end{equation*}
	where $dV_{g_{E}}$ is the unique Riemannian volume form of $(\mathbb{T}^{3},g_{E})$.	The following notations define the relevant norms we use to measure the small data solutions.
	The mean $ \bar{f} $ of a function $ f $ is given by
	\begin{equation*}
		\bar{f}=\frac{\int_{\mathbb{T}^{3}}f}{\int_{\mathbb{T}^{3}}}.
	\end{equation*}
	By $ \|\cdot\|_{H^{s}} $, we denote the standard Sobolev norm of order $ s $, i.e.,
	\begin{equation*}
		\|f\|_{H^{s}}^{2}=\sum_{|k|\leq s}\int_{\mathbb{T}^{3}}|\partial^{k}f|^{2}
	\end{equation*}
	where 
	\begin{equation*}
		|\partial^{k}f|\coloneqq \sqrt{\delta^{i_{1}j_{1}}\cdots\delta^{i_{k}j_{k}}\partial_{i_{1}}\cdots \partial_{i_{k}}f\partial_{j_{1}}\cdots\partial_{j_{k}}f}.
	\end{equation*}
	In general, the $ L^{2} $-inner product of tensors $ T_{1} $ and $ T_{2} $, e.g.,  given by
	\begin{equation*}
		T_{1}=(T_{1})^{ab}{}_{cd}, \qquad T_{2}=(T_{2})_{ab}{}^{c}{}_{d},
	\end{equation*}
	is assumed to be contracted via the Euclidean metric, e.g.,
	\begin{equation}\label{eq:innerproduct}
		(T_{1},T_{2})_{L^{2}}=\int \delta^{j_{a}i_{b}}(T_{1})^{i_{1}i_{2}}{}_{j_{1}j_{a}}(T_{2})_{i_{1}i_{2}}{}^{j_{1}}{}_{i_{b}}.
	\end{equation}
	
	\begin{remark}\label{remark:ambiguity}
		Note that the $L^{2}$-inner product defined in \eqref{eq:innerproduct} is \emph{unambiguous}. The two tensors in the product need to have the same number of ordered indices. Each slot on of the first tensor is contracted with the corresponding slot of the second tensor. Should these indices both be contra-variant or covariant, then a contraction via the Euclidean metric is assumed. 
	\end{remark}
	
	\subsection{Newton-Cartan gravity}
	The canonical approach to formulating a Newtonian cosmological model of a fluid evolving under the effect of gravity is \red{either via a direct change of coordinates based on the length-scaling (e.g., \cite{Baumann})} or via the Newton-Cartan-\blue{Ehlers} model of gravity \cite{brauer1994}. \red{Both lead to the same system, and so we just recall the latter approach.} \red{The Newton-Cartan-\blue{Ehlers} equations} take the form of the following hyperbolic-elliptic system 
	\begin{equation}\label{eq:newtoncartan}
		\begin{cases}
			\partial_{t}\rho+\partial_{i}(\rho u^{i})+\theta\rho=0,\\
			\partial_{t}u^{i}+u^{j}\partial_{j}u^{i}+\rho^{-1}h^{ij}\partial_{j}p+\frac{2}{3}\theta u^{i}-h^{ij}\partial_{j}\phi =0,\\
			h^{ij}\partial_{i}\partial_{j}\phi=-4\pi (\rho-\bar{\rho}),\\
			p(\rho)=K\rho.
		\end{cases}
	\end{equation}
	In \eqref{eq:newtoncartan}, $ \rho $ and $ u $ are the density and velocity of the fluid, respectively. The function $ \theta $ depends only on time and plays the roll of the Hubble-parameter. The spatial domain of the problem is the Riemannian manifold $ (\mathbb{T}^{3},h) $, where $ h $ is uniquely determined by the system
	\begin{equation}\label{eq:evolutionh}
		\begin{cases}
			\partial_{t}h_{ij}=\frac{2\theta}{3}h_{ij},\\
			h(t_{0})_{ij}=\delta_{ij}.
		\end{cases}
	\end{equation}
	In addition to the two hyperbolic and the elliptic equations in \eqref{eq:newtoncartan}, a linear equation of state is chosen to close the system. For additional discussion and motivation of \eqref{eq:newtoncartan}, we refer to \cite{brauer1994,brauer1992,gong2024}. 
\subsection{Decoupling the expansion rate}	
The aim of this paper is to find the precise relation between global regularity of the fluid and the expansion rate of space. To obtain the freedom to prescribe the expansion rate we decouple \eqref{eq:newtoncartan} from \eqref{eq:evolutionh} and fix the spatial metric as follows:

\begin{equation}
\begin{aligned}
h(t)&=a^2(t)\delta,\\
\theta&=3\frac{\dot a }{a},\\
a(t)&=t^\alpha.
\end{aligned}
\end{equation}	
	
This leads to the decoupled system
	\begin{equation}\label{eq:systemrhou}
		\begin{cases}
			\partial_{t}U+A^{k}(t,U)\partial_{k}U=F(t,U,\phi),\\
			a^{-2}\delta^{ij}\partial_{i}\partial_{j}\phi=-4\pi (\rho-\bar{\rho}),\\
			p(\rho)=K\rho,
		\end{cases}
	\end{equation}
	where $ U=(\rho,u) $ is the solution vector and 
	\begin{equation}
		\begin{aligned}
		&A^{k}(t,U)=\begin{pmatrix}
			u^{k} & \rho \delta^{k}_{j}\\
			\frac{K}{a^{2}\rho}\delta^{ki} & u^{k}\delta^{i}_{j}
		\end{pmatrix},
		&F(t,U,\phi)=\begin{pmatrix}
			-3\frac{\dot{a}}{a}\rho \\
			-2\frac{\dot{a}}{a}u^{i}+a^{-2}\delta^{ij}\partial_{j}\phi
		\end{pmatrix}.
		\end{aligned}
	\end{equation}

The system in \eqref{eq:systemrhou} has the following homogeneous solutions, with $a(1)=1$ and $\phi=0$:
\begin{equation}
\begin{aligned}
\rho_{\mathrm{hom}}(t)&=\rho_{\mathrm{hom}}(1)a(t)^{-3},\\
u^i_{\mathrm{hom}}(t)&=u^i_{\mathrm{hom}}(1)a^{-2}(t).
\end{aligned}
\end{equation}
	Introducing expansion-normalized variables $ (L(t,U),v(t,U)) $ via
	\begin{equation}\label{eq:Lv}
		\begin{aligned}
			&L(t,U)=\log(a(t)^{3}\rho), & v^{i}(t,U)=a(t)u^{i},
		\end{aligned}
	\end{equation}
	we find that
	\begin{equation}\label{eq:systemLv}
		\begin{cases}
			\partial_{t}V+B^{k}(t,V)\partial_{k}V=G(t,V,\phi),\\
			a^{-2}\delta^{ij}\partial_{i}\partial_{j}\phi=-4\pi (\rho-\bar{\rho}),\\
			p(\rho)=K\rho,
		\end{cases}
	\end{equation}
	where $ V=(L,v) $ is the solution vector and 
	\begin{equation}
		\begin{aligned}
			&B^{k}(t,V)=a^{-1}\begin{pmatrix}
				v^{k} &  \delta^{k}_{j}\\
				K\delta^{ki} & v^{k}\delta^{i}_{j}
			\end{pmatrix},
			&G(t,V,\phi)=\begin{pmatrix}
				0 \\
				-\frac{\dot{a}}{a}v^{i}+a^{-1}\delta^{ij}\partial_{j}\phi
			\end{pmatrix}.
		\end{aligned}
	\end{equation}
This system has in particular the homogeneous solution $L(t)=L(1)$, $v^i=0$ and $\phi=0$, whose perturbations we study further below.

	\begin{remark}
		The system given in \eqref{eq:systemLv} consists of a quasilinear symmetric-hyperbolic and a linear elliptic part. For sufficiently regular initial data, one can prove a local well-posedness result via standard iteration arguments and elliptic estimates. For details, we refer to \cite{brauer1992}. 
	\end{remark}

\section{Stability of homogeneous solutions for $\alpha>2/3$}\label{sec:stability}
This section contains the proof of global existence for small data as formulated in Theorem \ref{thm1}.

\subsection{Setup and preliminary estimates}	
	We study the hyperbolic-elliptic system of PDEs given by
	\begin{equation}\label{eq:equationsofmotion}
		\begin{cases}
			\partial_{t}L=-t^{-\alpha}v^{i}\partial_{i}L-t^{-\alpha}\partial_{i}v^{i},\\
			\partial_{t}v^{i}=-t^{-\alpha}K\delta^{ij}\partial_{j}L-t^{-\alpha}v^{j}\partial_{j}v^{i}-\alpha t^{-1}v^{i}+t^{-\alpha}\delta^{ij}\partial_{j}\phi,\\
			t^{-2\alpha}\laplacian \phi = -(\rho-\bar{\rho}).
		\end{cases}
	\end{equation}
	In \eqref{eq:equationsofmotion}, $ K\in (0,\frac{1}{3}) $, $ \alpha\in (0,\infty) $, $ (t,x)\in ([t_{0},t_{1})\times \mathbb{T}^{3}) $ and $ \rho \coloneqq t^{-3\alpha}\exp(L) $. For the entirety of this article, $ s $ will denote an integer with $ s\geq 3 $. We are studying solutions to the system \eqref{eq:equationsofmotion} that have the regularity properties
	\begin{align*}
		L,v^{i} \in C^{0}([0,T];H^{s}(\mathbb{T}^{3}))\cap C^{1}([0,T];H^{s-1}(\mathbb{T}^{3})),\\
		\phi \in C^{0}([0,T];N^{s+1}(\mathbb{T}^{3}))\cap C^{1}([0,T];N^{s}(\mathbb{T}^{3})),
	\end{align*}
	where, for $ k\in \mathbb{N} $,
	\begin{equation}
		N^{k}(\mathbb{T}^{3})\coloneqq \left\{f\in H^{k}(\mathbb{T}^{3})\biggr\vert \int_{\mathbb{T}^{3}}f=0\right\}. 
	\end{equation}
	By standard elliptic theory, we find that for $ \ell\geq 1 $
	\begin{equation}\label{eq:estimatephiH2}
		\|\partial \phi\|_{H^{\ell}}\leq Ct^{2\alpha}\|\rho-\bar{\rho}\|_{H^{\ell-1}}.
	\end{equation}

	From the equations of motion \eqref{eq:equationsofmotion}, we immediately derive
	estimates for the mean $ \bar{v} $, 
	\begin{equation}\label{eq:estimatemeanv}
		\frac{d}{dt}\bar{v}^{i}=-\alpha t^{-1}\bar{v}^{i}-t^{-\alpha}(v,\partial v^{i})_{L^{2}}\leq -\alpha t^{-1}\bar{v}^{i}+t^{-\alpha}\|v\|_{L^{2}}\|v\|_{H^{1}}.
	\end{equation}
	Similarly, we get an estimate on the mean value of $ L $,
	\begin{equation}\label{eq:estimatemeanL}
		\frac{d}{dt}\bar{L}=-t^{-\alpha}\int v^{a}\partial_{a}L\lesssim t^{-\alpha}\|L\|_{\dot{H}^{1}}\|v\|_{L^{2}}.
	\end{equation}
	This, in turn, allows for pointwise control of $ L $, as via Sobolev embedding and Poincar\'e inequality,
	\begin{equation}\label{eq:estimateLpointwise}
		\|L\|_{L^{\infty}}\lesssim \|L\|_{H^{2}}\lesssim \bar{L}+\|L\|_{\dot{H}^{1}}+\|L\|_{\dot{H}^{2}}.
	\end{equation}
		
	\subsubsection{Definition of the energy functionals}
	
	Before deriving a priori-estimates for Sobolev-type energies of the fluid variables $ (L,v) $, we will define energy functionals that are coercive and adapted to the equations in \eqref{eq:equationsofmotion}. 
	
	\begin{definition}
		We define the first-order energy functional $ E_{1} $ via
		\begin{equation}\label{eq:energyfirstorder}
			E_{1}[v,L]\coloneqq |\bar{v}|^{2}+ \|v\|_{\dot{H}^{1}}^{2}+K\|L\|_{\dot{H}^{1}}^{2}+t^{-1+\alpha}c(v,\partial L)_{L^{2}},
		\end{equation}
		where $ c>0 $ is some constant to be determined at a later point. 
		Now let $ \ell>1 $ be an integer. We define the following homogeneous energy of order $ \ell $
		\begin{equation}\label{eq:higherorderenergy}
			E_{\ell}[v,L]\coloneqq \|v\|_{\dot{H}^{\ell}}^{2}+K\|L\|_{\dot{H}^{\ell}}^{2}+t^{-1+\alpha}c(\partial^{\ell-1}v,\partial^{\ell}L)_{L^{2}}.
		\end{equation}
		In addition, we define the total energy of order $ s $, $ \mathcal{E}_{s} $, via
		\begin{equation*}
			\mathcal{E}_{s}[v,L]\coloneqq \sum_{k=1}^{s} E_{k}[v,L].
		\end{equation*}
	\end{definition}
	
	\begin{lemma}[Coercivity]\label{lemma:coercivity}
		For all $ s\geq 3 $
		\begin{equation*}
			\mathcal{E}_{s}[v,L]^{\frac{1}{2}}\simeq \|v\|_{H^{s}}+\|\partial L\|_{H^{s-1}}.
		\end{equation*}
	\end{lemma}
	\begin{proof}
		First, we take a careful look at the $ L^{2} $-norm. Note that, using the Cauchy-Schwarz inequality, we have that
		\begin{equation*}
			|\bar{v}|=\frac{1}{\int_{\mathbb{T}^{3}}}\left|\int v \right| \leq \frac{1}{\left(\int_{\mathbb{T}^{3}}\right)^{\frac{1}{2}}}\|v\|_{L^{2}}.
		\end{equation*}
		Using this, and estimating the mixed term in \eqref{eq:higherorderenergy} by the Cauchy-Schwarz and Young inequalities, it is straightforward to see that
		\begin{equation*}
			\mathcal{E}_{s}[v,L]\lesssim  \|v\|_{H^{s}}^{2}+\|L\|_{H^{s}}^{2}.
		\end{equation*}
		Regarding the coercivity over the Sobolev-norm, first consider that, by the Poincar\'e inequality, we have that
		\begin{equation*}
			\|v\|_{L^{2}}^{2}=\int |v-\bar{v}+\bar{v}|^{2}\lesssim |\bar{v}|^{2}+\|\partial v\|_{L^{2}}^{2}.
		\end{equation*} 
		When estimating $ \|v\|_{H^{s}}^{2}+\|\partial L\|_{H^{s-1}}^{2} $, the time-dependent coefficient $ t^{-1+\alpha} $ in front of the mixed term in \eqref{eq:higherorderenergy} assures that it can always be absorbed into the norm. With these considerations in mind, it is easy to see that
		\begin{equation*}
			  \|v\|_{H^{s}}^{2}+\|\partial L\|_{H^{s-1}}^{2} \lesssim \mathcal{E}_{s}[v,L].
		\end{equation*}
	\end{proof}	
		
	Note that $ E_{1} $ includes the mean velocity $ \bar{v} $. The specific structure of the norm $ \mathcal{E}_{s} $ is adapted to the problem at hand and will be motivated by the following discussion. 	
	
	\subsection{Lower-order estimates}
	We now present the correction mechanism, which provides the essential energy estimate for the proof, at the level of $H^1$-regularity. This is the key idea relevant for the proof of Theorem \ref{thm1}. Further below, we present estimates in higher regularity, which follow the same structure. 	
	
	\subsubsection{Evolution of first derivatives}
	
	Commuting the evolution equations for $ (L,v)^{T} $ in \eqref{eq:equationsofmotion} with the spatial derivative operator $ \partial_{j} $ gives 
	\begin{equation}\label{eq:eomfirstderivatives}
		\begin{aligned}
			\partial_{t}\partial_{j}L&=-t^{-\alpha}v^{a}\partial_{j}\partial_{a}L-t^{-\alpha}\partial_{j}\partial_{a}v^{a}-t^{\alpha}\partial_{a}L\partial_{j}v^{a},\\
			\partial_{t}\partial_{j}v^{i}&=-t^{-\alpha}K\delta^{ia}\partial_{j}\partial_{a}L-t^{-\alpha}v^{a}\partial_{j}\partial_{a}v^{i}-\alpha t^{-1}\partial_{j}v^{i}+t^{-\alpha}\delta^{ia}\partial_{j}\partial_{a}\phi-t^{-\alpha}\partial_{a}v^{i}\partial_{j}v^{a}.
		\end{aligned}
	\end{equation}
	The $ \dot{H}^{1} $-energy of $ v $ is given by
	\begin{equation*}
		\|v\|_{\dot{H}^{1}}^{2}=\int |\partial v|^{2}.
	\end{equation*}
	A simple calculation gives
	\begin{equation}\label{eq:dervfirstorder}
		\begin{aligned}
			\partial_{t}\|v\|_{\dot{H}^{1}}^{2}&=2\int \delta_{kl}\delta^{ij}\partial_{i}v^{k}\left(-t^{-\alpha}K\delta^{la}\partial_{j}\partial_{a}L-t^{-\alpha}v^{a}\partial_{j}\partial_{a}v^{l}-\alpha t^{-1}\partial_{j}v^{l}+t^{-\alpha}\delta^{la}\partial_{j}\partial_{a}\phi-t^{-\alpha}\partial_{a}v^{l}\partial_{j}v^{a}\right)\\
			&=-2\alpha t^{-1}\|v\|_{\dot{H}^{1}}^{2}-2Kt^{-\alpha}(\partial v,\partial^{2}L )_{L^{2}}-2t^{-\alpha}(v\partial v,\partial^{2}v)_{L^{2}}\\
			&\quad -2t^{-\alpha}\int \delta_{kl}\delta^{ij}\partial_{i}v^{k}\partial_{j}v^{a}\partial_{a}v^{l}+2t^{-\alpha}(\partial v,\partial^{2}\phi)_{L^{2}}.
		\end{aligned}
	\end{equation}
	A similar calculation for $ L $ yields 
	\begin{equation}\label{eq:derLfirstorder}
		\begin{aligned}
			\partial_{t}\|L\|_{\dot{H}^{1}}^{2}&=2\int \delta^{ij}\partial_{i}L\left(-t^{-\alpha}v^{a}\partial_{j}\partial_{a}L-t^{-\alpha}\partial_{j}\partial_{a}v^{a}-t^{\alpha}\partial_{a}L\partial_{j}v^{a}\right)\\
			&=-2t^{-\alpha}(v\partial L,\partial^{2}L)_{L^{2}}-2t^{-\alpha}\int \delta^{ij}\partial_{i}L\partial_{j}\partial_{a}v^{a}-2t^{-\alpha}(\partial L\partial L,\partial v)_{L^{2}}.
		\end{aligned}
	\end{equation}
	In addition, we investigate the evolution of the mixed term, i.e.,
	\begin{equation}\label{eq:derLmixirstorder}
		\begin{aligned}
			\partial_{t}(v,\partial L)_{L^{2}}&=\int \partial_{i}L\left(-t^{-\alpha}K\delta^{ia}\partial_{a}L-t^{-\alpha}v^{a}\partial_{a}v^{i}-\alpha t^{-1}v^{i}+t^{-\alpha}\delta^{ia}\partial_{a}\phi\right)\\
			&\quad +\int v^{j}\left(-t^{-\alpha}v^{a}\partial_{j}\partial_{a}L-t^{-\alpha}\partial_{j}\partial_{a}v^{a}-t^{\alpha}\partial_{a}L\partial_{j}v^{a}\right)\\
			&=-Kt^{-\alpha}\|L\|_{\dot{H}^{1}}^{2}-\alpha t^{-1}(v,\partial L)_{L^{2}}-t^{-\alpha}(v\partial L,\partial v)_{L^{2}}+t^{-\alpha}(\partial L,\partial \phi)_{L^{2}}\\
			&\quad-t^{-\alpha}(v^{2},\partial^{2}L)_{L^{2}}-t^{-\alpha}\int v^{j}\partial_{j}\partial_{a}v^{a}-t^{-\alpha}(v\partial L,\partial v)_{L^{2}}.
		\end{aligned}
	\end{equation}

	\subsubsection{Estimates on the first-order energy}
	
	Due to the special composition of $ E_{1} $ and to illustrate the fundamental idea of the proof we now present a detailed account of the decay mechanism for the first-order energy.

	\begin{lemma}[First-order energy estimate]\label{lemma:firstorder}
		There exists a monotonously increasing function $ C $ with $ C(1)=1 $ and a constant $ C^{\prime}>0 $, such that the first-order energy defined in \eqref{eq:energyfirstorder} with $ c=\alpha $ enjoys the estimate
		\begin{equation}\label{eq:energyinequalityfirstorder}
			\frac{d}{dt}E_{1}[v,L](t)\leq -\alpha t^{-1} E_{1}[v,L]+C^{\prime}\left(t^{-\alpha}\|\partial v\|_{L^{\infty}}E_{1}[v,L]+\left(1+C(\|L\|_{L^{\infty}})\right)(t^{-1-\alpha}+t^{-3\alpha})E_{1}[v,L]\right).
		\end{equation}
	\end{lemma}
	\begin{proof}
		Combining the results from \eqref{eq:estimatemeanv}, \eqref{eq:dervfirstorder}, \eqref{eq:derLfirstorder}, \eqref{eq:derLmixirstorder}, we find that
		\begin{equation}\label{eq:firstorderidentity}
			\begin{aligned}
			\frac{d}{dt}E_{1}[v,L]&= -\alpha t^{-1}\left(2|\bar{v}|^{2}+2\|v\|_{\dot{H}^{1}}^{2}+c\frac{K}{\alpha}\|L\|_{\dot{H}^{1}}^{2}+\frac{c}{\alpha}t^{-1+\alpha}(v,\partial L)_{L^{2}}\right) \\
			&\quad +2t^{-\alpha}(\partial v,\partial^{2}\phi)_{L^{2}}+ct^{-1}(\partial L,\partial \phi)_{L^{2}}-ct^{-1}\int v^{j}\partial_{j}\partial_{a}v^{a}\\
			&\quad +I_{1}+I_{2}+I_{3},
			\end{aligned}
		\end{equation}
		where 
		\begin{align*}
			I_{1}&=-2Kt^{-\alpha}(\partial v,\partial^{2}L )_{L^{2}}-2Kt^{-\alpha}\int \delta^{ij}\partial_{i}L\partial_{j}\partial_{a}v^{a},\\
			I_{2}&=-2t^{-\alpha}(v\partial v,\partial^{2}v)_{L^{2}}-2t^{-\alpha}(v\partial L,\partial^{2}L)_{L^{2}}-ct^{-1}(v^{2},\partial^{2}L)_{L^{2}},\\
			I_{3}&=-2t^{-\alpha}\delta_{ij}\bar{v}^{i}\int v^{k}\partial_{k}v^{j}-2t^{-\alpha}\int \delta_{kl}\delta^{ij}\partial_{i}v^{k}\partial_{j}v^{a}\partial_{a}v^{l}-2Kt^{-\alpha}(\partial L\partial L,\partial v)_{L^{2}}\\
			&\quad -2ct^{-1}(v\partial L,\partial v)_{L^{2}}.
		\end{align*}
		From \eqref{eq:firstorderidentity} we see that the evolution of $ E_{1} $ is governed by some damping terms as well as a variety of other terms. We will now proceed to show that the latter can be treated as a small error. 
		
		First, we need to make sure that no regularity is lost. Hence, we closely inspect terms that include $ \partial^{2}L $ or $ \partial^{2}v $ which are, with the exception of one explicit term in \eqref{eq:firstorderidentity}, contained in $ I_{1} $ and $ I_{2} $. An application of integration by parts shows that, in fact, $ I_{1}=0 $. 
		
		Inspecting $ I_{2} $, we see that, after integrating by parts, the last term does not pose a threat with regards to regularity. For the first term, we note that
		\begin{equation*}
			(v\partial v,\partial^{2}v)_{L^{2}}=-(v\partial v,\partial^{2}v)_{L^{2}}-\int \partial_{i}v^{i}\delta^{ab}\delta_{kl}\partial_{a}v^{k}\partial_{b}v^{l},
		\end{equation*}
		again, simply integrating by parts. Remember that the contraction in this inner product is unambiguous, see Remark \ref{remark:ambiguity}. The same argument works for $ (v\partial L,\partial^{2}L)_{L^{2}} $. Using crude estimates, we conclude that
		\begin{equation*}
			|I_{2}|\lesssim t^{-\alpha}\|\partial v\|_{L^{\infty}}E_{1}[v,L]. 
		\end{equation*}
		Inspecting the terms in $ I_{3} $, we see that straightforward applications of Sobolev and Hölder inequalities as well as elementary estimates lead to a similar conclusion, i.e., 
		\begin{equation*}
				|I_{3}|\lesssim t^{-\alpha}\|\partial v\|_{L^{\infty}}E_{1}[v,L]. 
		\end{equation*}
		
		Lastly, we turn to the terms in \eqref{eq:firstorderidentity} involving the potential $ \phi $. Using the estimate in \eqref{eq:estimatephiH2}, we find 
		\begin{equation*}
			t^{-\alpha}(\partial v,\partial^{2}\phi)_{L^{2}}\leq t^{-\alpha}\|\partial v\|_{L^{2}}\|\partial^{2} \phi\|_{L^{2}}\lesssim t^{\alpha}\|\partial v\|_{L^{2}}\|\rho-\bar{\rho}\|_{L^{2}}.
		\end{equation*}
		Furthermore, by definition of $L$ and using the Poincar\'e inequality, we find
		\begin{equation*}
			\|\rho-\bar{\rho}\|_{L^{2}}\lesssim t^{-3\alpha}\|\partial \exp(L)\|_{L^{2}}\lesssim t^{-3\alpha}(1+C(\|L\|_{L^{\infty}}))\|L\|_{\dot{H}^{1}}.
		\end{equation*}
		Similarly, we estimate
		\begin{equation*}
			t^{-1}(\partial L,\partial \phi)_{L^{2}}\lesssim t^{-1-\alpha} \|L\|_{\dot{H}^{1}}^{2}(1+C(\|L\|_{L^{\infty}})).
		\end{equation*}
		Finally, using integration by parts, we find that
		\begin{equation*}
			-\alpha t^{-1}\int v^{j}\partial_{j}\partial_{a}v^{a}=\alpha \|v\|_{\dot{H}^{1}}^{2},
		\end{equation*}
		which exactly takes away a factor $ \alpha $ from our decay inducing term in \eqref{eq:firstorderidentity}, yielding the desired result. 
	\end{proof}

	\begin{remark}
		Of course, a priori, the estimate in \eqref{eq:energyinequalityfirstorder} does not close, as we do not have control over the factors $ \|\partial v\|_{L^{\infty}} $ and $ \|\partial L\|_{L^{\infty}} $. However, said factors appear at first order, even in the evolution for the higher-order energies. Hence, under suitable bootstrap assumptions, we can use Sobolev embedding to estimate these terms and close the estimates. 
	\end{remark}
	
	\subsection{Higher-order estimates}
	
	We start by commuting the evolution equations of the system \eqref{eq:equationsofmotion} with a derivative operator $$ \partial^{I}=\partial_{1}^{i_{1}}\cdots \partial_{n}^{i_{n}}, $$ where $ I=(i_{1},\dots, i_{n}) $ is a multi-index. 
	\begin{equation}\label{eq:eomhigherorder}
		\begin{aligned}
			\partial_{t}\partial^{I}L&=-t^{-\alpha}v^{a}\partial_{a} \partial^{I}L-t^{-\alpha}\partial_{a}\partial^{I}v^{a}+t^{-\alpha}\mathfrak{F}^{I}_{L},\\
			\partial_{t}\partial^{I}v^{i}&=-Kt^{-\alpha}\delta^{ia}\partial_{a}\partial^{I}L-t^{-\alpha}v^{a}\partial_{a}\partial^{I}v^{i}-\alpha t^{-1}\partial^{I}v^{i}+t^{-\alpha}\delta^{ia}\partial^{I}\partial_{a}\phi+t^{-\alpha}\mathfrak{F}^{i,I}_{v}
		\end{aligned}
	\end{equation}
	where 
	\begin{equation}
		\begin{aligned}
			\mathfrak{F}^{I}_{L}&=-\left[\partial^{I},v^{a}\partial_{a}\right]L,\\
			\mathfrak{F}_{v}^{i,I}&=-\left[\partial^{I},v^{a}\partial_{a}\right]v^{i}.
		\end{aligned}
	\end{equation}

	\subsubsection{Evolution of crucial higher-order terms}
	
	Recall that we set $ s $ an integer, satisfying $ s\geq 3 $. We start with the evolution of the $ \dot{H}^{s} $-energy of the velocity. Using \eqref{eq:eomhigherorder}, we find
	\begin{align*}
		\frac{d}{dt}\|v\|_{\dot{H}^{s}}^{2}&=2\int \delta^{i_{1}j_{1}}\cdots\delta^{i_{s}j_{s}} \delta_{ij}\partial_{i_{1}}\cdots \partial_{i_{s}}v^{j}\partial_{t}\partial_{j_{1}}\cdots \partial_{j_{s}}v^{i}\\
		&=-2\alpha t^{-1}\|v\|_{\dot{H}^{s}}^{2}-2Kt^{-\alpha}(\partial^{s}v,\partial^{s+1}L)_{L^{2}}-2t^{-\alpha}(v\partial^{s}v,\partial^{s+1}v)_{L^{2}}\\
		&\quad +2t^{-\alpha}(\partial^{s}v,\partial^{s+1}\phi)+2t^{-\alpha}\int  \delta_{ij}\partial_{i_{1}}\cdots \partial_{i_{s}}v^{j}\mathfrak{F}^{i,I(i_{1},\dots, i_{s})}_{v},
	\end{align*}
	where $ I(i_{1},\dots,i_{s}) $ is the associated multi-index. Similarly, for $ L $, we find that
	\begin{align*}
		\frac{d}{dt}\|L\|_{\dot{H}^{s}}^{2}&=2\int \delta^{i_{1}j_{1}}\cdots\delta^{i_{s}j_{s}} \partial_{i_{1}}\cdots \partial_{i_{s}}L\partial_{t}\partial_{j_{1}}\cdots \partial_{j_{s}}L\\
		&=-2t^{-\alpha}\int \delta^{i_{1}j_{1}}\cdots\delta^{i_{s}j_{s}} \partial_{i_{1}}\cdots \partial_{i_{s}}L\partial_{a}\partial_{j_{1}}\cdots\partial_{j_{n}}v^{a}\\
		&\quad -2t^{-\alpha}(v\partial^{s}L,\partial^{s+1}L)_{L^{2}}+2t^{-\alpha}\int  \partial_{i_{1}}\cdots \partial_{i_{s}}L\mathfrak{F}^{I(i_{1},\dots, i_{s})}_{L}.
	\end{align*}
	
	We also calculate the evolution of the mixed term, 
	
	\begin{align*}
		\frac{d}{dt}(\partial^{s-1}v,\partial^{s}L)_{L^{2}}&=-t^{-\alpha}(v\partial^{s-1}v,\partial^{s+1}L)_{L^{2}}+t^{-\alpha}\int  \delta_{ij}\partial_{i_{1}}\cdots \partial_{i_{s-1}}v^{i}\mathfrak{F}^{I(i_{1},\dots, i_{s-1},j)}_{L}\\
		&\quad-t^{-\alpha}\int \delta^{i_{1}j_{1}}\cdots \delta^{i_{s-1}j_{s-1}}\partial_{i_{1}}\cdots \partial_{i_{s-1}}v^{i}\partial_{a}\partial_{j_{1}}\cdots \partial_{j_{s-1}}\partial_{i}v^{a}\\
		&\quad -Kt^{-\alpha}\|L\|_{\dot{H}^{s}}^{2}-t^{-\alpha}(\partial^{s}v,v\partial^{s}L)_{L^{2}}-\alpha t^{-1}(\partial^{s-1}v,\partial^{s}L)_{L^{2}}\\
		&\quad + t^{-\alpha}(\partial^{s}\phi,\partial^{s}L)_{L^{2}}+t^{-\alpha}\int  \partial_{i_{1}}\cdots \partial_{i_{s}}L\mathfrak{F}^{i_{s},I(i_{1},\dots, i_{s-1})}_{v}.
	\end{align*}

	\subsubsection{Preliminary estimates of higher-order terms}
	
	Before deriving an energy estimate for higher-order energies, we recall some basic inequalities from function space theory. 
	
	\begin{lemma}[Moser-type estimates]\label{lemma:Moser}
		Suppose that $\ell>1$ is an integer and $\beta$ a multi-index with $|\beta|=\ell$. Furthermore, assume that $f$ and $g$ are functions with $f,g\in H^{\ell}\cap W_{1,\infty}$. Then
		\begin{equation*}
			\|[\partial^{\beta},f\partial ]g\|_{L^{2}}\lesssim \|\partial f\|_{L^{\infty}}\|\partial g\|_{H^{\ell-1}}+\|\partial g\|_{L^{\infty}}\|\partial f\|_{H^{\ell-1}}.
		\end{equation*}
		If additionally, $F\in C^{\infty}$ with $F(0)=0$ and $\ell>\frac{d}{2}$, there exists a continuous function $C:[0,\infty)\to [0,\infty)$, such that
		\begin{equation*}
			\|F(f)\|_{H^{\ell}}\leq C(\|f\|_{L^{\infty}})\|f\|_{H^{\ell}}. 
		\end{equation*}
	\end{lemma}
	For a proof of Lemma \ref{lemma:Moser}, see, e.g., \cite[Appendix C]{serre2007}.

	\begin{lemma}\label{lemma:higherorderestimate}
		There exists a monotonously increasing function $ C $ with $ C(1)=1 $ and a constant $ C^{\prime}>0 $, such that the higher-order energies defined in \eqref{eq:higherorderenergy} satisfy the bound
		\begin{align*}
			\frac{d}{dt}\mathcal{E}_{s}[v,L]&\leq -\alpha t^{-1}\mathcal{E}_{s}[v,L]+C^{\prime}(t^{-\alpha}\mathcal{E}_{s}[v,L]^{\frac{3}{2}}+t^{-2\alpha}\mathcal{E}_{s}[v,L])\\
			&\quad+t^{-2\alpha}C(1+\|L\|_{\infty})\left(\|v\|_{H^{s}}\|\partial L\|_{H^{s-1}}+t^{-1+\alpha}\|L\|_{H^{s}}\|\partial L\|_{H^{s-2}}\right).
		\end{align*}
	\end{lemma}

	\begin{proof}
		Without loss of generality, we investigate the evolution of the top-order energy $ E_{s} $. The arguments can be applied verbatim to $ E_{k} $, for $1<  k< s $. The cancellation of terms of order $ s+1 $ is due to integration by parts, as is the case in the first-order energy. The decay inducing terms are of a similar form to the first-order case, and the argument regarding the decay rate is analogous. The term involving
		\begin{equation*}
			(v\partial^{s}v,\partial^{s+1}v)_{L^{2}}
		\end{equation*}
		can simply be integrated by parts. Hence, an application of Sobolev embedding yields that
		\begin{equation*}
			(v\partial^{s}v,\partial^{s+1}v)_{L^{2}}\lesssim \|\partial v\|_{L^{\infty}}\|\partial v\|_{H^{s-1}}^{2}\leq \mathcal{E}_{s}^{\frac{3}{2}}.
		\end{equation*}
		The same is true for $ (v\partial^{s}L,\partial^{s+1}L)_{L^{2}} $. Terms like 
		\begin{equation*}
			(\partial^{s}v,v \partial^{s}L)_{L^{2}}
		\end{equation*}
		are purely perturbative and can be estimated roughly in the same manner. The only terms still left to analyze are those involving the potential $ \phi $ and the commutator terms. We start with the latter. We use standard Moser-type estimates and apply the Sobolev inequality to find
		\begin{equation*}
			\|\mathfrak{F}_{v}^{I(i_{1},\dots,i_{s})}\|_{L^{2}}\lesssim \|\partial v\|_{L^{\infty}}\|v\|_{H^{s}}\lesssim \mathcal{E}_{s}.
		\end{equation*}
		Note that we have used that
		\begin{equation*}
			\|v\|_{L^{2}}\lesssim |\bar{v}|+\|\partial v\|_{L^{2}}\leq 2E_{1}^{\frac{1}{2}}\leq 2\mathcal{E}_{s}^{\frac{1}{2}}.
		\end{equation*}
		Regarding $ \mathfrak{F}_{L} $, we have to be a bit more careful, as we do not have strong control over $ |\bar{L}| $ the same way we do regarding $ |\bar{v}| $. However, the Moser-type estimate for the commutator in Lemma \ref{lemma:Moser} gives us
		\begin{equation*}
			\|\mathfrak{F}_{L}^{I(i_{1},\dots,i_{s})}\|_{L^{2}}\lesssim \|\partial v\|_{L^{\infty}}\|\partial L\|_{H^{s-1}}+\|\partial L\|_{L^{\infty}}\|v\|_{H^{s}}\lesssim \mathcal{E}_{s}.
		\end{equation*}
		This shows that all terms involving the commutators are purely perturbative of order $ \mathcal{E}^{\frac{3}{2}} $. Additionally, we estimate
		\begin{align*}
			(\partial^{s}v,\partial^{s+1}\phi)_{L^{2}}&\leq \| v\|_{\dot{H}^{s}}\|\partial^{s+1}\phi\|_{L^{2}} \lesssim t^{2\alpha}\|v\|_{\dot{H}^{s}}\|\rho-\bar{\rho}\|_{H^{s-1}}
			\lesssim t^{-\alpha}\|v\|_{\dot{H}^{s}}\|\partial \exp(L)\|_{H^{s-2}}\\
			&\lesssim t^{-\alpha}\|v\|_{\dot{H}^{s}}\|\exp(L)\partial L\|_{H^{s-2}}.
		\end{align*}
	Again, using Moser estimates, this yields 
	\begin{equation*}
		(\partial^{s}v,\partial^{s+1}\phi)_{L^{2}}\lesssim t^{-\alpha}C(1+\|L\|_{\infty})\|v\|_{H^{s}}\|\partial L\|_{H^{s-1}},
	\end{equation*}
	where $ C(\cdot) $ is monotonously increasing with $ C(1)=1 $. Similarly, for the other term involving $ \phi $, we have that
	\begin{align*}
		(\partial^{s}\phi,\partial^{s}L)_{L^{2}}&\leq \|L\|_{H^{s}}\|\partial^{s}\phi\|_{L^{2}}
		\lesssim t^{2\alpha}\|L\|_{H^{s}}\|\rho-\bar{\rho}\|_{H^{s-2}}\\
		&\lesssim t^{-\alpha}\|L\|_{H^{s}}\|\partial \exp(L)\|_{H^{s-3}}\\
		&\lesssim t^{-\alpha}\|L\|_{H^{s}}\|\exp(L)\partial L\|_{H^{s-3}}\\
		&\lesssim t^{-\alpha}C(1+\|L\|_{\infty})\|L\|_{H^{s}}\|\partial L\|_{H^{s-1}}.
	\end{align*}
	\end{proof}
	
	\subsection{Stability proof}
	We formulate the main theorem in the following and present the proof based on the foregoing lemmas subsequently.
	\begin{theorem}\label{thm1}
		Let $ s\in \mathbb{Z}_{\geq 3} $, $ \alpha>\frac{2}{3} $ and $ L_{0}\in \mathbb{R} $. Then, there exists a $ \delta >0 $, such that for all initial data $ (\mathring{L},\mathring{v})\in H^{s}(\mathbb{T}^{3})\times H^{s}(\mathbb{T}^{3}) $ satisfying 
		\begin{equation*}
			\|\partial (\mathring{L}-L_{0})\|_{H^{s-1}}+\|\mathring{v}\|_{H^{s}}<\delta,
		\end{equation*}
		there exists a unique, global solution 
		\begin{equation*}
			 (L,v,\phi)\in C\left([t_{0},\infty);(H^{s}(\mathbb{T}^{3}))^{3}\right)\cap C^{1}\left([t_{0},\infty); (H^{s-1}(\mathbb{T}^{3}))^{3}\right)
		\end{equation*}
		to the system in \eqref{eq:equationsofmotion}, such that
		\begin{equation*}
			L\Big|_{t=t_{0}}=\mathring{L}, \quad v\Big|_{t=t_{0}}=\mathring{v}.
		\end{equation*}
	\end{theorem}
	\begin{proof}
		By standard existence theory for hyperbolic-elliptic systems, see, e.g., \cite{brauer1992}, we know that \eqref{eq:equationsofmotion} has a local solution. Given a solution $ (L,v) $, by continuity, we know that for all $ \varepsilon >0 $, there exists a $ t_{*}>t_{0} $ such that a solution satisfies 
		\begin{equation}\label{eq:bootstrap}
			\|\partial (L(t,\cdot)-L_{0}(\cdot))\|_{H^{s-1}}+\|v(t,\cdot)\|_{H^{s}} \leq \varepsilon t^{-\frac{\alpha}{2}+\frac{\eta}{2}},
		\end{equation}
		for some small $ \eta>0 $ to be determined later and for all $ t\in [t_{0},t_{*}) $, provided $ \delta $ is sufficiently small. Using Lemma \ref{lemma:coercivity}, this implies that there exists a constant $ c>0 $, such that 
		\begin{equation*}
			\mathcal{E}_{s}[v,L](t)\leq c^{2} \varepsilon^{2} t^{-\alpha+\eta}
		\end{equation*}
		for all $ t\in [t_{0},t_{*}) $. Combining Lemma \ref{lemma:higherorderestimate} and Lemma \ref{lemma:firstorder}, we find that for $ t\in [t_{0},t_{*}) $
		\begin{equation*}
			\frac{d}{dt}\mathcal{E}_{s}\leq -\alpha t^{-1}\mathcal{E}_{s} +C(1+\|L\|_{\infty})t^{-2\alpha}\mathcal{E}_{s}+C^{\prime}(t^{-\alpha}\mathcal{E}_{s}^{\frac{3}{2}}),
		\end{equation*}
		for some function $ C(\cdot) $ and $ C^{\prime}>0 $. Introducing the rescaled energy $ \tilde{\mathcal{E}}_{s}=t^{\alpha-\eta}\mathcal{E}_{s} $. we immediately read off
		\begin{equation*}
			\frac{d}{dt}\tilde{\mathcal{E}}_{s}\leq -\eta t^{-1}\tilde{\mathcal{E}}_{s} +C(1+\|L\|_{\infty})t^{-2\alpha}\tilde{\mathcal{E}}_{s}+C^{\prime}(t^{-\frac{3}{2}\alpha+\frac{\eta}{2}}\tilde{\mathcal{E}}_{s}^{\frac{3}{2}}),
		\end{equation*}
		Furthermore, by \eqref{eq:estimatemeanL}, we see that
		\begin{equation*}
			\frac{d}{dt}\bar{L}\leq C^{\prime \prime}t^{-\alpha}\mathcal{E}_{s} \leq C^{\prime \prime}C\varepsilon^{2} t^{-2\alpha +\eta}.
		\end{equation*}
		Integrating this equations, we see that, provided $ -2\alpha+\eta<-1 $, we have an upper bound on the size of $ |\bar{L}| $. 
		Hence, all terms on the right-hand side of \eqref{eq:estimateLpointwise} are bounded by a constant, and thus
		\begin{equation*}
			C(1+\|L\|_{L^{\infty}})\leq C^{\prime\prime\prime},
		\end{equation*}
		for all $ t\in [t_{0},t_{*}) $. Combining these considerations, we find that on all of $ t\in [t_{0},t_{*}) $,
		\begin{equation*}
			\frac{d}{dt}\tilde{\mathcal{E}}_{s}\leq \left(C^{\prime\prime\prime}t^{-2\alpha}+cC^{\prime}\varepsilon t^{-\frac{3}{2}\alpha+\frac{\eta}{2}}\right)\tilde{\mathcal{E}}_{s}.
		\end{equation*}
		An application of Grönwall's lemma yields
		\begin{equation*}
			\tilde{\mathcal{E}}_{s}[v,L](t)\leq c^{2} \delta^{2}\exp\left(\int_{t_{0}}^{t^{*}}C^{\prime\prime\prime}t^{-2\alpha}+cC^{\prime}\varepsilon t^{-\frac{3}{2}\alpha+\frac{\eta}{2}}dt\right). 
		\end{equation*}
		Provided $ \delta $ is sufficiently small, this implies a strict improvement over the bootstrap assumption \eqref{eq:bootstrap}. It follows that the bound on the norm is extendable and hence, by the continuation principle the solution is global. For a textbook treatment of the continuation principle for hyperbolic systems see, e.g., \cite{serre2007}. 
	\end{proof}
	
	\begin{remark}
	Note that the sign of the force term in the Euler equation does not affect the arguments in the proof. In consequence, Theorem \ref{thm1} also holds for the Euler-Poisson system with repulsive forces.
	\end{remark}

	\section{Numerical analysis of the Euler equation on expanding backgrounds}
	\label{sec:numerics}
	In this section we study the Euler equations in symmetry reduction to an effective 1+1-dimensional system on an expanding background, i.e., the expansion-normalized variables only depend on one spatial coordinate, $L=L(t,x)$, $v=v(t,x)$. Furthermore, as outlined in the introduction, we neglect the gravitational interaction via the potential $ \phi $. The reduced system then reads
	\begin{equation}\label{eq:pureEuler}
		\begin{cases}
			\partial_{t}L=-t^{-\alpha}v\partial_{x} L-t^{-\alpha}\partial_{x} v ,\\
			\partial_{t}v=-t^{-\alpha}K\partial_{x} L-t^{-\alpha}v \partial_{x} v-\alpha t^{-1}v.\\
		\end{cases}
	\end{equation}
We study the numerical evolution of initial data at $t=1$ of the specific form
\begin{equation}\label{init}
(L,v)|_{t=1}=(0,\varepsilon\cdot\sin(x))
\end{equation}
for $\varepsilon>0$. In the following, we denote by $H_{k}$ the energy functional of order $k\geq 0$,
\begin{equation}
	H_{k}[f](t)\coloneqq \|\partial^{k} f(t,\cdot)\|_{L^{2}}^{2}.
\end{equation}
\subsection{Stable region}
Numerical data confirms the result of the analysis of the previous section. For $\alpha>2/3$ the normalized Sobolev norm $H_4(t)/H_4(1)$ decays in time. More specifically for the choice $K=1/6$ and $\alpha=0.8$ the following figure 
shows the time development of $H_4(t)/H_4(1)$ on the left and $[H_4(t)/H_0(t)]/[H_4(1)/H_0(1)]$ on the right.
\begin{figure}[ht]
		\centering
		\includegraphics[scale=0.8]{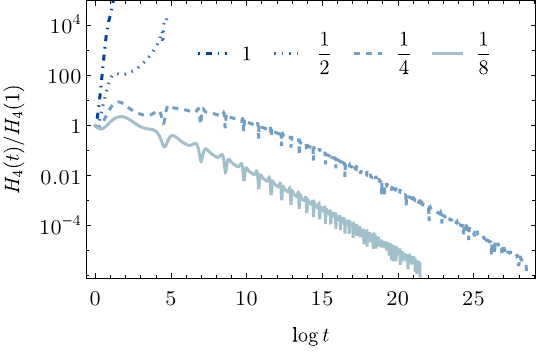}
		\includegraphics[scale=0.8]{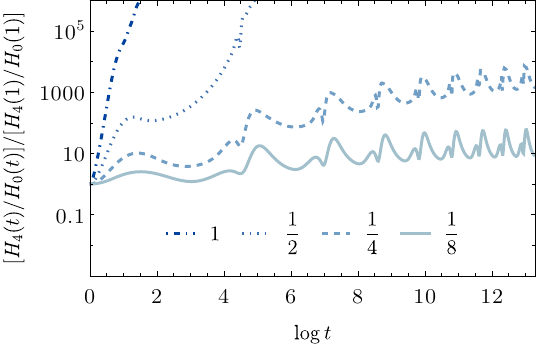}
		\caption{Time-evolution of the Sobolev norms of solutions with initial data of type \eqref{init} with $\varepsilon$ of size as indicated in the respective legend. The data is taken for the specific case of $\alpha=0.8$ and $K=1/6$.}
		\label{fig1}
	\end{figure}
In the left picture the decay is more apparent. The numbers associated to the different graphs correspond to the amplitude $\varepsilon$ of the initial data. The threshold for detecting stability was set to $10^{-6}$.

\subsection{Unstable region}
For $\alpha<2/3$ we study the growth of $H_4(t)/H_0(t)$ normalized by the initial value. It is shown in Figure \ref{fig2}, where $K=1/6$ and $\alpha=0.3$, the legend shows $\varepsilon$. The threshold to detect instability was set to $10^6$.
\begin{figure}[ht]
		\centering
		\includegraphics[scale=0.8]{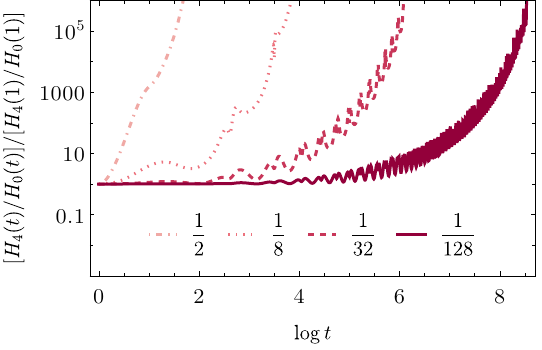}
		\caption{Time-evolution of the Sobolev norms of solutions with initial data of type \eqref{init} with $\varepsilon$ of size as indicated in the respective legend. The data is taken for the specific case of $\alpha=0.3$ and $K=1/6$.}
		\label{fig2}
	\end{figure}

\subsection{Scaling analysis and detection of the critical line}

We apply a scaling analysis, as described in the following, to identify $\alpha=2/3$ as the critical expansion rate.
For fixed $K\leq1/3$ and $\alpha\leq 2/3$ we generate sequence of solutions corresponding to a sequence of initial data \eqref{init} with decreasing $\varepsilon>0$. To each $\varepsilon$ corresponds a time of shock formation $t_*$ defined here according to the instabilty threshold by

\begin{equation}
H_4(t_*)/H_0(t_*)\geq 10^6\cdot H_4(1)/H_0(1).
\end{equation}
\begin{figure}[ht]
		\centering
		\includegraphics[scale=0.55]{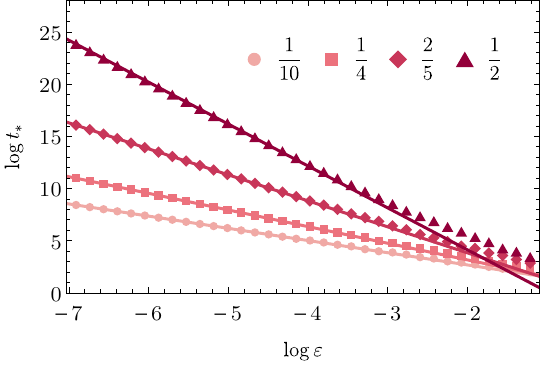}
			\includegraphics[scale=0.55]{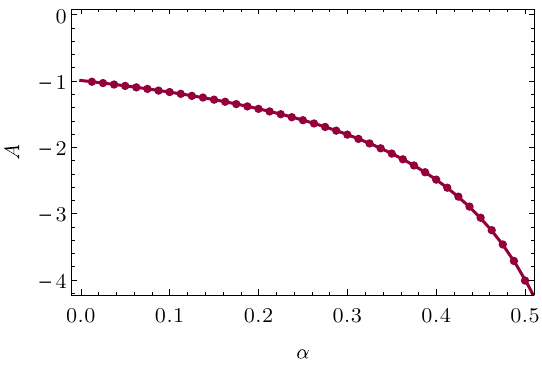}
		\includegraphics[scale=0.55]{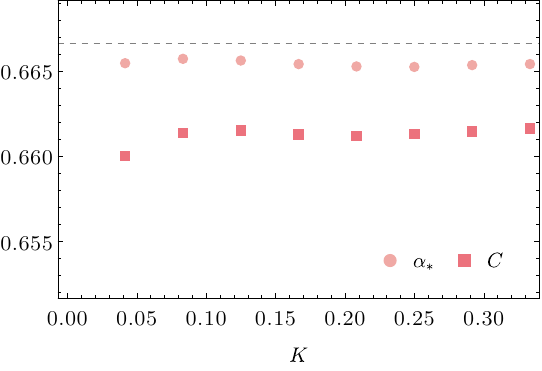}
		\caption{Left: Breaking time depending on size of the initial data for different expansion rates. Center: slope of the breaking time function depending on the expansion rate. Right: Critical expansion rate for different values of $K$.}
		\label{fig3}
	\end{figure}
	
The data, as shown in the leftmost graph in Figure \ref{fig3}, shows the relation

\begin{equation}
\log t_*= A\cdot \log\varepsilon + B
\end{equation}
for suitable constants $A$ and $B$. The legend shows $\alpha$, while $K=1/6$. Other values of $K$ give similar data.

In the next step, the value of the slope $A$ can be read off from the data and can be plotted against the value of the expansion rate, as done in the middle graph of Figure \ref{fig3}. This shows a relation of the form

\begin{equation}
A=\frac{C}{\alpha-\alpha_{\mathrm{crit}}},
\end{equation}
which characterizes the critical expansion rate $\alpha_{\mathrm{crit}}$. A fit of the data for $K=1/6$ (and similar for other admissible values of $K$) yields the following numerical values for $C$ and $\alpha_{\mathrm{crit}}$.

\begin{equation}
\begin{aligned}
C&\approx 0.663 \,...\, 0.665\\
\alpha_{\mathrm{crit}}&\approx 0.666
\end{aligned}
\end{equation}
To test the independence of $K$ cf.~the rightmost plot in Figure \ref{fig3}. The dependence of the time of shock formation on the size of initial data can then be determined (under the assumption of $C=\alpha_{\mathrm{crit}}=2/3$) to
\begin{equation}
t_{\mathrm{crit}}\sim \varepsilon^{-1/(1-3\alpha/2)}.
\end{equation}

	\subsection{Riemann invariants and characteristics }
	
	Finally, we take the complementary perspective on shock formation from the point of view of the physical characteristics of the system.
	We visualize the geometric blowup of solutions to the Euler equations by simulations using the Riemann invariants and plot the corresponding characteristics. The $[0,2\pi]$-domain is repetitively extended for the purpose of the presentation. We stress that this is done to present the reader with a geometric visualization of the shock phenomenon of the previous section, rather than to derive precise blowup times. 
	
	\subsection{Equations of motion of the Riemann-invariants}
		
	The symmetry-reduced equations of motion in the unrescaled fluid coordinates $ (u,\rho) $ take the form
		\begin{equation}\label{eq:eomunrescaled}
					\begin{cases}
						\rho_{t}+u\rho_{x}+\rho u_{x}=-3\frac{\dot{a}}{a}\rho,\\
						u_{t}+\frac{K}{a^{2}\rho}\rho_{x}+uu_{x}=-2\frac{\dot{a}}{a}u+a^{-2}\phi_{x},\\
						a^{-2}\laplacian \phi = -(\rho-\bar{\rho}).
						\end{cases}
			\end{equation}
	
		The coefficient matrix associated with the spatial derivative operator is given by
		\begin{equation*}
				A(t,x,u)=\begin{pmatrix}
						u & \rho \\ \frac{K}{a^{2}\rho} & u
					\end{pmatrix}
			\end{equation*}
		with left eigenvectors $ \ell_{\pm} $ and eigenvalues $ \lambda_{\pm} $ given by 
		\begin{align}
				&\ell_{\pm}=\left(\pm\frac{\sqrt{K}}{\rho},a\right), 
				&\lambda_{\pm}=\pm\frac{\sqrt{K}}{a}+u.
			\end{align}	
		Hence, the Riemann invariants $ R_{\pm} $ are given by
		\begin{equation*}
				R_{\pm}(\rho,u)=\pm \log \rho^{\sqrt{K}} + a u,
			\end{equation*}
		which is straightforward to see, calculating $ \partial R_{\pm} $. Thus,
		\begin{align*}
				u&=\frac{1}{2a}(R_{+}+R_{-}),\\
				\rho&=\exp\left(\frac{1}{2\sqrt{K}}(R_{+}-R_{-})\right).
			\end{align*}
		The PDEs satisfied by $ R_{\pm} $ are given by 
		\begin{equation}\label{eq:Riemann}
				\begin{aligned}
					(\partial_{t}+\lambda_{+}\partial_{x})R_{+}(\rho,u)&=-3\sqrt{K}\frac{\dot{a}}{a}-2\dot{a}u+a^{-1}\partial_{x}\phi,\\
					(\partial_{t}+\lambda_{-}\partial_{x})R_{-}(\rho,u)&=+3\sqrt{K}\frac{\dot{a}}{a}-2\dot{a}u+a^{-1}\partial_{x}\phi.
					\end{aligned}
			\end{equation}
		Hence, along the integral curves $ \gamma_{\pm} $ with characteristic speeds $ \lambda_{\pm} $, i.e.,
		\begin{equation}\label{eq:characteristics}
				\frac{d}{dt}\gamma_{\pm}(t)=\lambda_{\pm}(\rho(t,\gamma_{\pm}(t)),u(t,\gamma_{\pm}(t))), \gamma_{\pm}(0)=\xi \in S_{1},
			\end{equation}
		the equations of motion \eqref{eq:Riemann} reduce to the set of ODEs
		\begin{equation}\label{eq:RiemannODE}
				\begin{aligned}
						\frac{d}{dt}R_{+}(\rho,u)\circ \gamma_{+}(t)&=-3\sqrt{K}\frac{\dot{a}}{a}(t)-2\dot{a}(t)u(t,\gamma_{+}(t))+a^{-1}(t)\partial_{x}\phi(t,\gamma_{+}(t)),\\
						\frac{d}{dt}R_{-}(\rho,u)\circ \gamma_{-}(t)&=+3\sqrt{K}\frac{\dot{a}}{a}(t)-2\dot{a}(t)u(t,\gamma_{-}(t))+a^{-1}(t)\partial_{x}\phi(t,\gamma_{-}(t)).
					\end{aligned}
	\end{equation}

	For the sake of simplicity, we will from now on drop the coupling to the gravitational potential $\phi$. Hence, the equations we simulate are given by the first-order PDE-system
	
	\begin{equation}\label{eq:Riemannnophi}
		\begin{aligned}
			(\partial_{t}+\lambda_{+}\partial_{x})R_{+}(\rho,u)&=-3\sqrt{K}\frac{\dot{a}}{a}-2\dot{a}u,\\
			(\partial_{t}+\lambda_{-}\partial_{x})R_{-}(\rho,u)&=+3\sqrt{K}\frac{\dot{a}}{a}-2\dot{a}u.
		\end{aligned}
	\end{equation}
	
	\subsection{Visualization of the characteristics}
	
	Similar to the previous section, we study solutions launched by initial data of the type
	\begin{equation}
		(L,v)|_{t=1}=(0,\varepsilon\cdot\sin(x)).
	\end{equation}

	\begin{figure}[ht!]
		\centering
		\includegraphics[scale=0.3]{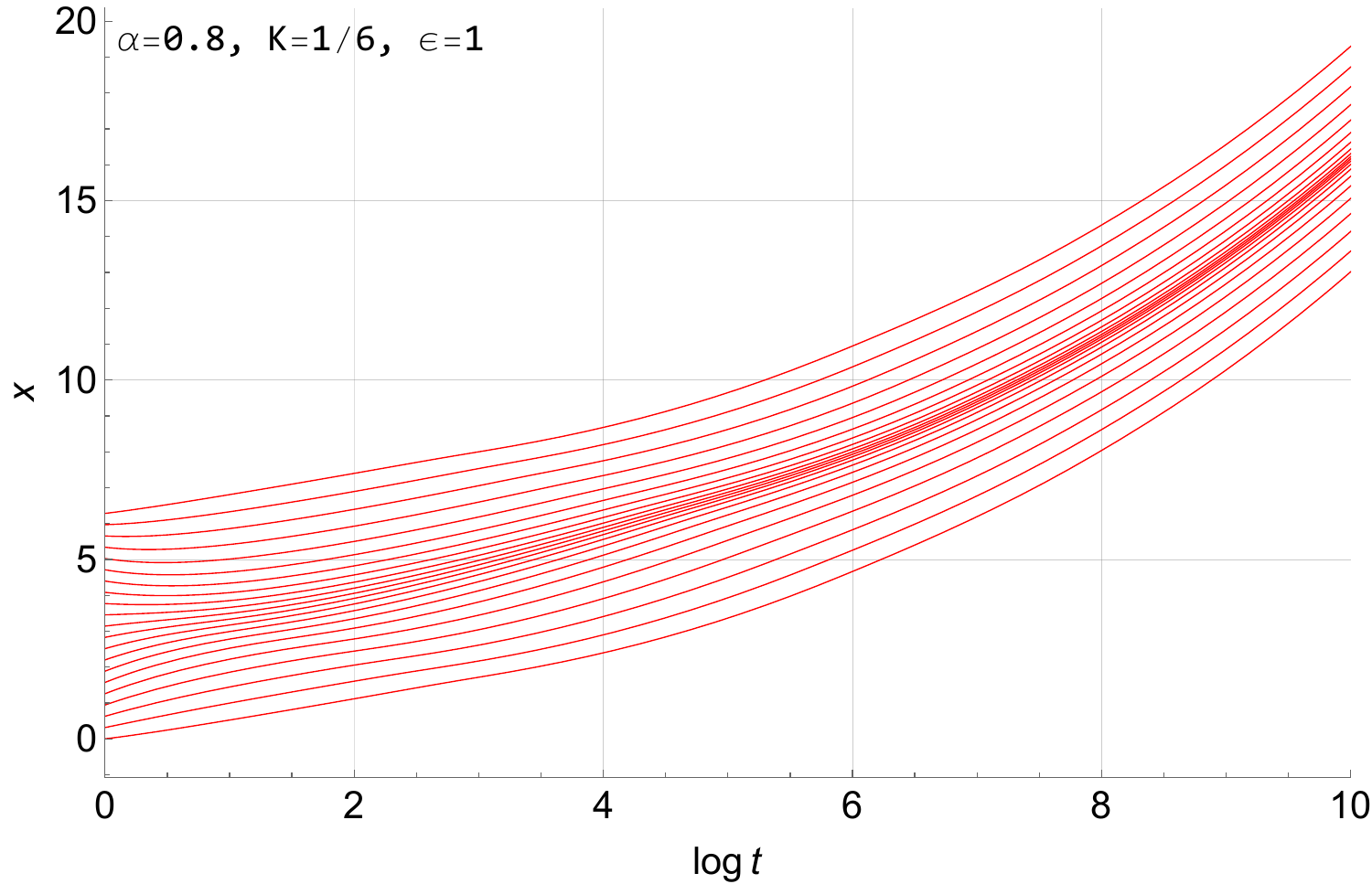}
		\includegraphics[scale=0.3]{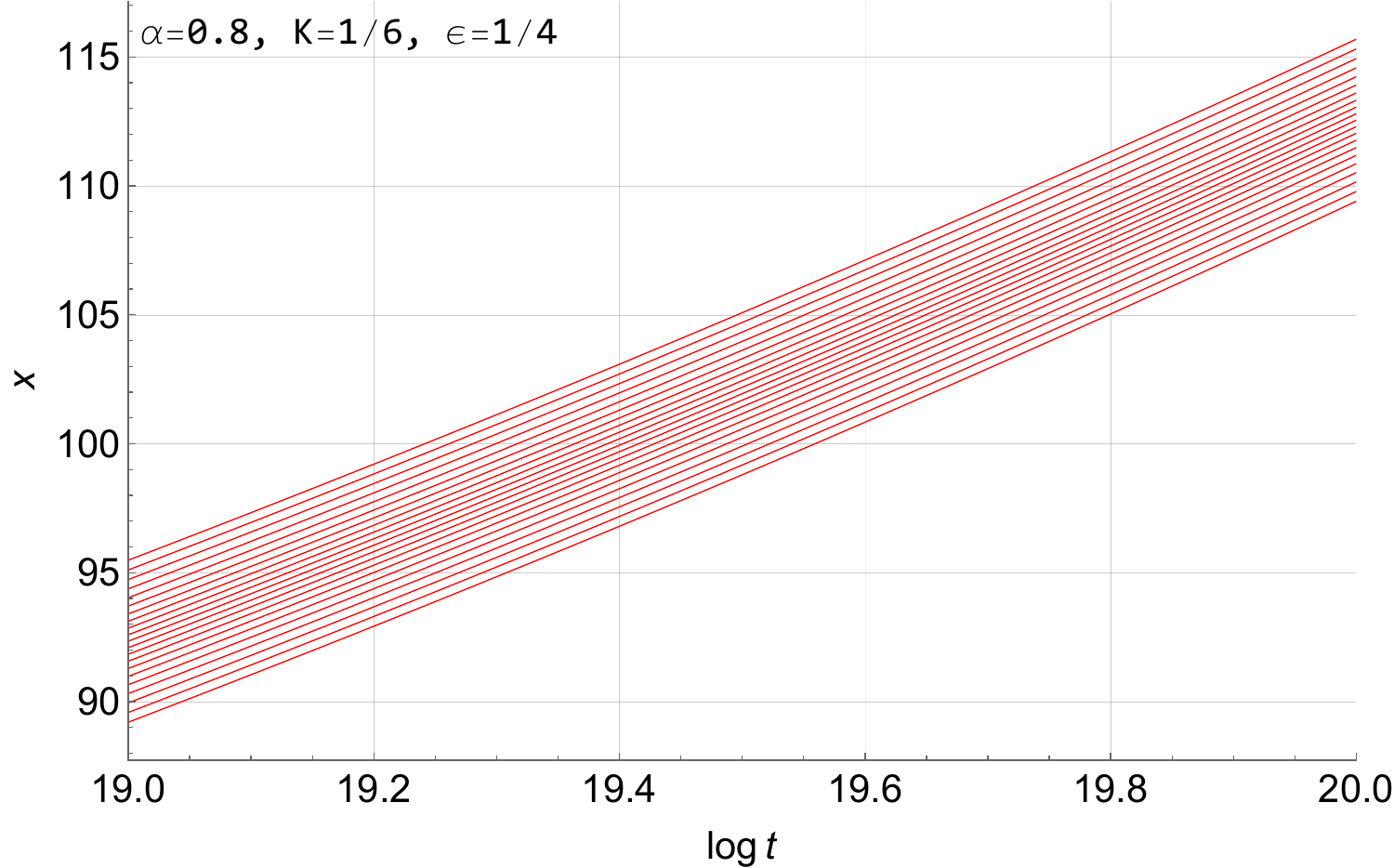}
		\caption{\label{fig5} These pictures depict the physical characteristics for $\alpha=0.8$ and $K=1/6$. The amplitude of the left and right plot is given by $\varepsilon=1$ and $\varepsilon=1/4$, respectively. Note that, for the sake of presentation, the flow of the characteristics is plotted in the universal covering space of $S_{1}$, i.e., $\mathbb{R}$.}
	\end{figure}
	
	Figure \ref{fig5} shows simulations for $\alpha=0.8$ and $K=1/6$. In the left picture, we see an amplitude of $\varepsilon=1$ where the characteristics accumulate slowly. In the right picture, which has amplitude $\varepsilon=1/4$, the density of the characteristics homogenizes rapidly. Note that, in the second picture, we are only presented with a short time frame in the later part of the evolution.

	\begin{figure}[ht!]
		\centering
		\includegraphics[scale=0.3]{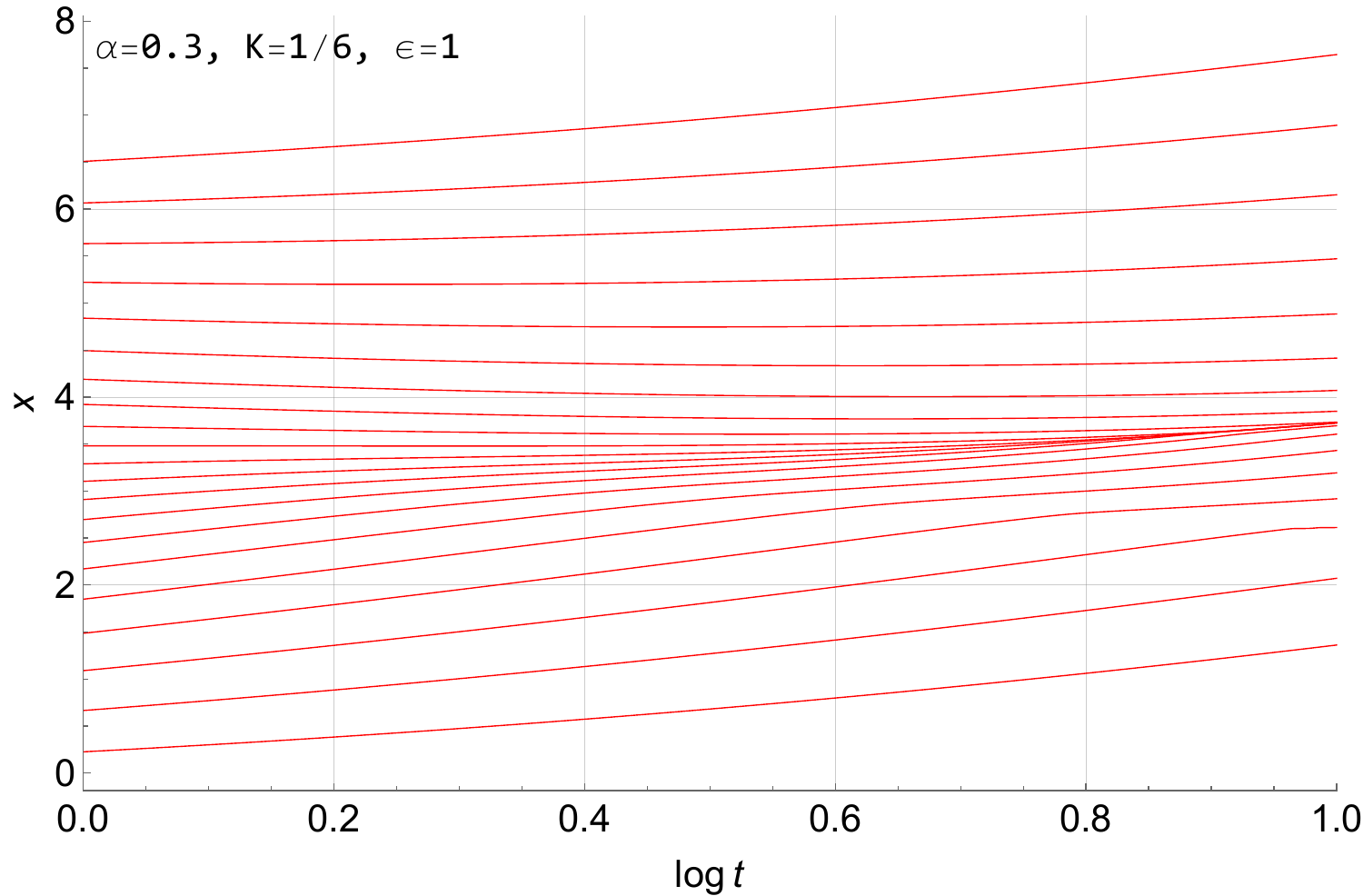}
		\includegraphics[scale=0.3]{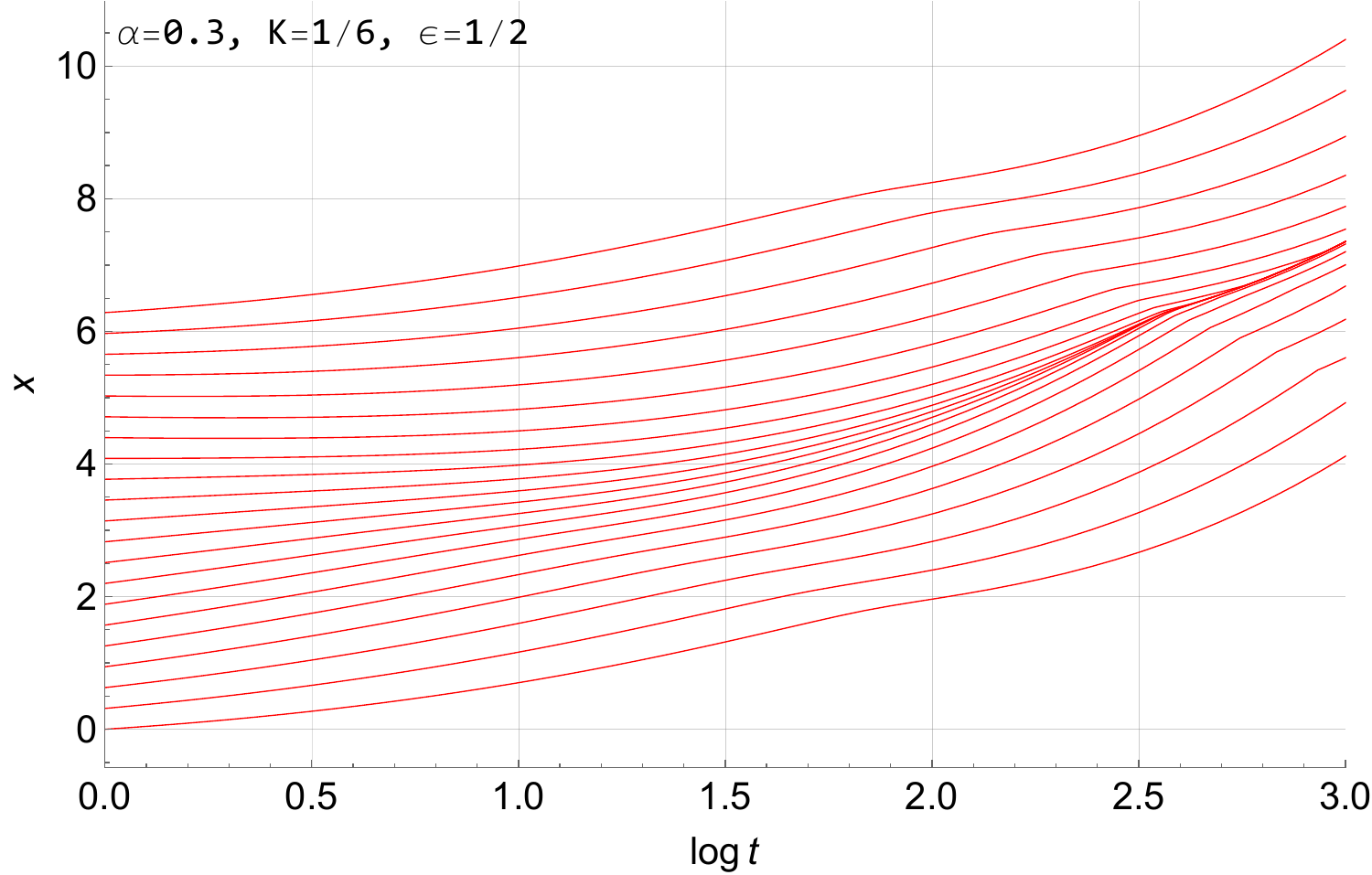}
		\includegraphics[scale=0.3]{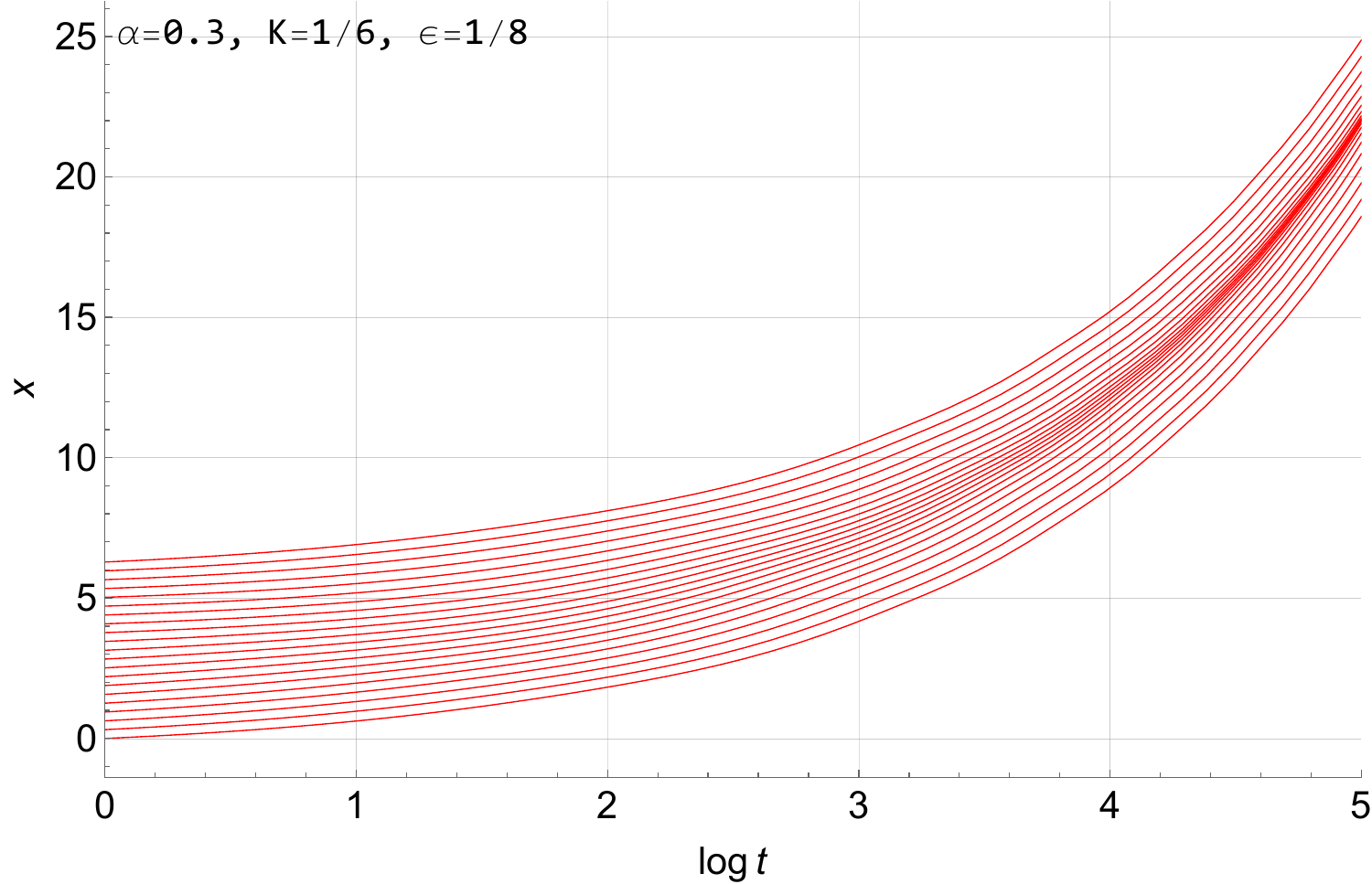}
		\includegraphics[scale=0.3]{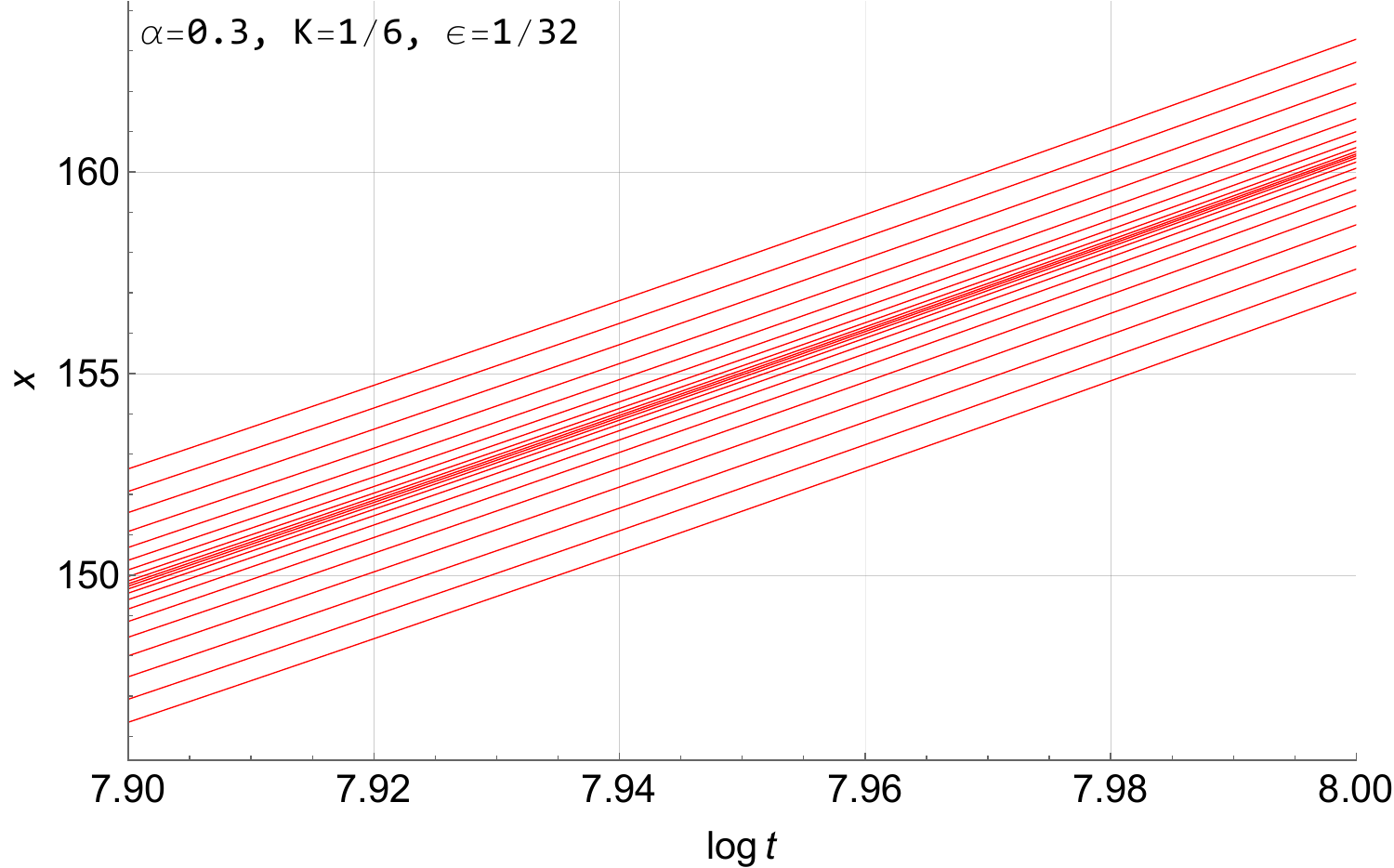}
		\caption{\label{fig6}These pictures depict the physical characteristics for $\alpha=0.3$ and $K=1/6$. The amplitude from left to right and top to bottom is given by $\varepsilon=1$, $\varepsilon=1/2$, $\varepsilon=1/8$ and $\varepsilon=1/32$, respectively. }
	\end{figure}

		In Figure \ref{fig6}, we see runs for $\alpha=0.3$. Top left and right show an amplitude of $\varepsilon=1$ and $\varepsilon=1/2$ in which the characteristics accumulate rapidly. In the bottom left plot, we see the characteristics for $\varepsilon=1/8$, which still concentrate fairly quickly. In the last picture we see the characteristics for an amplitude of $1/32$. There is a clear blowup of the density of the characteristics, but it is rather slow. Note that this is only a short time interval at the end of the simulation.

\subsection{Conclusion of the numerical experiments}
We have performed the numerical experiments for the Euler equations without the additional coupling to the potential. In the presence of a potential the numerical data shows a less definitive behaviour in the unstable regime with the present techniques. Nevertheless, for the Euler equations the numerical study demonstrates  that the critical threshold for stabilization is indeed located at $\alpha_{\mathrm{crit}}=2/3$. This matches with the analytical results in Theorem 1, which indicates that the analytical result is sharp as well as that the numerical studies provide accurate data.

\section{Outlook}
The stability result presented in Theorem \ref{thm1} considers solutions to the Euler-Poisson system on expanding tori \red{as regularly used in cosmology}. The details of the proof show that the force term caused by the gravitational potential in the Euler equation decays sufficiently fast and hence does not significantly affect the dynamics of the fluid in the small data regime. In particular, its sign is not relevant in the proof. This implies that Theorem \ref{thm1} also holds for the case of repulsive forces between the fluid particles. Whether repulsion indeed affects shock formation in the unstable regime is an interesting open question, which would require a more refined numerical study, which is now the subject of a study in progress.
	
\printbibliography

\begin{tabular}[h]{lll}
David Fajman \& Maximilian Ofner, &Maciej Maliborski &Todd Oliynyk \\
Faculty of Physics, &Institute of Analysis &School of Mathematics  \\ 
University of Vienna, &and Scientific Computing&9 Rainforest Walk \\
Boltzmanngasse 5, &Technical University Vienna&Monash University VIC 3800 \\
1090 Vienna, Austria. &Wiedner Hauptstraße 8-10,
&Australia \\
David.Fajman@univie.ac.at \& &1040 Vienna, Austria&todd.oliynyk@monash.edu     \\
maximilian.ofner@univie.ac.at &maciej.maliborski@tuwien.ac.at &\\
&&\\
  Zoe Wyatt&& \\
Department of Pure Mathematics &&  \\ 
 and Mathematical Statistics,&& \\
 Wilberforce Road,&&  \\
 Cambridge, CB3 0WB, U.K.&& \\
 zoe.wyatt@maths.cam.ac.uk  &&     \\

\end{tabular}

\end{document}